\documentclass[11pt]{article}

\usepackage{times,latexsym,graphicx,amsmath,amsfonts,amsthm}
\usepackage{natbib}

\newtheoremstyle{localthm}
	{5pt} 
	{5pt} 
	{\sl} 
	{} 
	{\bf} 
	{{\rm.}} 
	{.7em} 
	{} 

\theoremstyle{localthm}
\newtheorem{Theorem}{Theorem}

\newtheorem{Lemma}[Theorem]{Lemma}

\newtheoremstyle{localrem}
	{5pt} 
	{5pt} 
	{\rm} 
	{} 
	{\bf} 
	{{\rm.}} 
	{.7em} 
	{} 

\theoremstyle{localrem}

\def\bs{\boldsymbol}

\def\Ex{\operatorname{\mathbb{E}}}
\def\Pr{\operatorname{\mathbb{P}}}

\def\R{\mathbb{R}}

\def\BB{\mathcal{B}}
\def\XX{\mathcal{X}}

\begin{document}

\addtolength{\baselineskip}{+.2\baselineskip}

\title{Honest Confidence Bands for Isotonic Quantile Curves}

\author{Lutz D\"umbgen, Lukas L{\"u}thi\\
University of Bern}

\date{\today}

\maketitle

\paragraph{Abstract.}
We provide confidence bands for isotonic quantile curves in nonparametric univariate regression with guaranteed given coverage probability. The method is an adaptation of the confidence bands of D\"{u}mbgen and Johns (2004) for isotonic median curves.



\section{Introduction}

Let $(X_1,Y_1), (X_2,Y_2), \ldots, (X_n,Y_n)$ be independent random pairs consisting of covariate values $X_i$ in some real interval $\XX$ and response values $Y_i \in \R$. We assume that $\Pr(Y_i \le y \,|\, X_i) = F(y \,|\, X_i)$ with unknown conditional distribution functions $F(\cdot \,|\, x)$, $x \in \XX$. For any fixed $\gamma \in (0,1)$ and $x \in \XX$, let $Q_\gamma(x)$ be a $\gamma$-quantile of $F(\cdot \,|\, x)$, that is,
\[
	F(Q_\gamma(x)- |\,x) \ \le \ \gamma \ \le \ F(Q_\gamma(x) \,|\, x) .
\]
Our goal is to compute a $(1 - \alpha)$-confidence band $(L,U)$ for $Q_\gamma$ under the sole assumption that $Q_\gamma : \XX \to \R$ is isotonic, i.e.\ non-decreasing. Precisely, we want to determine functions $L = L(\cdot, \mathrm{data})$ and $U = U(\cdot, \mathrm{data})$ on $\XX$ with values in $[-\infty,\infty]$ such that
\[
	\Pr(L \le Q_\gamma \le U \ \text{on} \ \XX) \ \ge \ 1 - \alpha
\]
whenever $Q_\gamma$ is isotonic. In Section~\ref{sec:Construction} we describe the construction of the bands. It is a simplified and generalized version of the confidence bands of \cite{Duembgen_Johns_2004} who treated the case $\gamma = 0.5$ only. The bands are also similar in spirit to the confidence bands of \cite{Dimitriadis_etal_2022} for univariate binary regression or the confidence bands of \cite{Yang_Barber_2019} for isotonic regression with subgaussian errors. An important property of the bands is that their computation requires only $O(n^2)$ steps. An alternative method to construct confidence bands for isotonic quantile curves has been proposed recently by \cite{Chatterjee_Sen_2021}, but it requires certain a priori knowlegde about the conditional distribution functions $F(\cdot \,|\, x)$, $x \in \XX$.

Section~\ref{sec:Asymptotics} presents some asymptotic properties of the bands as the sample size $n$ tends to infinity. The results are similar to the results of \cite{Duembgen_Johns_2004} and \cite{Moesching_Duembgen_2020}. The latter paper treats point estimation of the conditional distribution functions $F(\cdot\,|\, x)$ under the assumption that they are isotonic in $x \in \XX$ with respect to the usual stochastic order. Section~\ref{sec:Examples} illustrates the bands with some numerical examples. There it is also illustrated what happens if the quantile functions are even assumed to be S-shaped, that is, isotonic and convex-concave, and if we replace our confidence band with the hull of all S-shaped functions fitting within the band. This stronger shape-constraint is plausible in many applications, see \cite{Feng_etal_2021} and the references cited therein. The latter paper treats least squares estimation of S-shaped mean functions.

\section{Construction of the bands}
\label{sec:Construction}

We condition on the observed values $X_1, \ldots, X_n$, so they become fixed real numbers, and the observations $Y_i$ are independent with distribution functions $F_i := F(\cdot \,|\, X_i)$. For an interval $B \subset \XX$, let
\[
	N(B) \ := \ \# \{i \colon X_i \in B\} .
\]
Suppose that $g : \XX \to \R$ is a candidate for $Q_\gamma$. If $N(B) > 0$ and $g = Q_\gamma$, then the distribution of the random sum
\[
	S_\ell(B,g) \ := \ \# \{i \colon X_i \in B, Y_i \le g(X_i)\}
\]
is stochastically greater than or equal to a binomial distribution with parameters $N(B)$ and $\gamma$. That is, for any integer $k \ge 0$,
\[
	\Pr \bigl( S_\ell(B,g) \le k \bigr) \ \le \ F_{N(B),\gamma}(k) ,
\]
where $F_{m,p}$ denotes the distribution function of the binomial distribution with parameters $m \in \mathbb{N}$ and $p \in [0,1]$. Analogously, the distribution of the random sum
\[
	S_u(B,g) \ := \ \# \{i \colon X_i \in B, Y_i \ge g(X_i)\}
\]
is stochastically greater than or equal to a binomial distribution with parameters $N(B)$ and $1 - \gamma$. For our confidence bands, we need an analogous statement simultaneously for all intervals $B \subset \XX$.

\begin{Lemma}
\label{lem:Coupling}
There exists a coupling between the observations $Y_1, \ldots, Y_n$ and stochastically independent Bernoulli variables $\xi_1, \ldots, \xi_n \in \{0,1\}$ with $\Ex(\xi_i) = \gamma$, such that for any version $g$ of $Q_\gamma$ and all intervals $B \subset \XX$,
\[
	S_\ell(B,g) \ \ge \ T_\ell(B) := \sum_{i\colon X_i \in B} \xi_i
	\quad\text{and}\quad
	S_u(B,g) \ \ge \ T_u(B) := \sum_{i\colon X_i \in B} (1 - \xi_i) .
\]
\end{Lemma}

\begin{proof}[\bf Proof of Lemma~\ref{lem:Coupling}]
Let $U_1, \ldots, U_n$ be independent random variables with uniform distribution on $(0,1)$. Then the quantile transformations
\[
	Y_i \ := \ F_i^{-1}(U_i) = \min\{y \in \R \colon F_i(y) \ge U_i\}
\]
yield independent random variables $Y_1, \ldots, Y_n$ with distribution functions $F_1,\ldots, F_n$, respectively. For any version $g$ of $Q_\gamma$ and $1 \le i \le n$, it follows from $F_i(g(X_i)) \ge \gamma$ that
\[
	1_{[Y_i \le g(X_i)]} \
	= \ 1_{[F_i(g(X_i)) \ge U_i]} \
	\ge \ 1_{[U_i \le \gamma]} .
\]
Similarly, the inequality $F_i(g(X_i)-) \le \gamma$ implies that
\[
	1_{[Y_i \ge g(X_i)]} \
	= \ 1_{[F_i(y) < U_i \ \text{for all} \ y < g(X_i)]} \
	\ge \ 1_{[F_i(g(X_i)-) < U_i]} \
	\ge \ 1_{[U_i > \gamma]} = 1 - 1_{[U_i \le \gamma]} .
\]
Consequently, $\xi_i := 1_{[U_i \le \gamma]}$ defines Bernoulli variables $\xi_1, \ldots, \xi_n$ with the desired properties.
\end{proof}

The confidence band can be constructed as follows. Let $\BB$ be a given family of compact intervals $B \subset \XX$ with endpoints in $\{X_1,\ldots,X_n\}$. Suppose that for any integer $m \in \{1,\ldots,n\}$, we have specified integers $c_\ell(m), c_u(m) \in \{0,1,\ldots,m\}$ such that
\begin{equation}
\label{ineq:conf.level}
	\Pr \bigl( T_\ell(B) \ge c_\ell(N(B)) \ \text{and} \ T_u(B) \ge c_u(N(B))
		\ \ \text{for all} \ B \in \BB \bigr) \
	\ge \ 1 - \alpha .
\end{equation}
Then we may claim with confidence $1 - \alpha$ that for any version $g$ of $Q_\gamma$ and all intervals $B \in \BB$,
\begin{equation}
\label{ineq:coverage}
	S_\ell(B,g) \ \ge \ c_\ell(N(B))
	\quad\text{and}\quad
	S_u(B,g) \ \ge \ c_u(N(B)) .
\end{equation}
Let $Y_{B:1} \le Y_{B:2} \le \cdots \le Y_{B:N(B)}$ be the order statistics of the observations $Y_i$ with $X_i \in B$, and let $Y_{B:0} := -\infty$, $Y_{B:N(B)+1} := \infty$. Then it follows from isotonicity of $g$ that
\[
	S_\ell(B,g) \ \ge \ c_\ell(N(B))
	\quad\text{implies that}\quad
	g(\max(B)) \ \ge \ Y_{B:c_\ell(N(B))} ,
\]
because $S_\ell(B,g) \le \# \bigl\{ i \colon X_i \in B, Y_i \le g(\max(B)) \bigr\}$. Likewise,
\[
	S_u(B,g) \ \ge \ c_u(N(B))
	\quad\text{implies that}\quad
	g(\min(B)) \ \le \ Y_{B:N(B)+1-c_u(N(B))} .
\]
Assuming that $Q_\gamma$ is isotonic, $Y_{B:c_\ell(N(B))}$ is a lower bound for all values $Q_\gamma(x)$, $x \ge \max(B)$, and $Y_{B:N(B)+1-c_i(N(B))}$ is an upper bound for all values $Q_\gamma(x)$, $x \le \min(B)$. These considerations lead to the following $(1 - \alpha)$-confidence band $(L,U)$ for $Q_\gamma$:
\begin{align}
\label{eq:Def.L}
	L(x) \
	:= \ &\sup_{B \in \BB\colon \max(B) \le x} \, Y_{B:c_\ell(N(B))} , \\
\nonumber
	= \ &\sup_{B \in \BB\colon \max(B) \le x} \,
		\min \bigl\{ r \in [-\infty,\infty) \colon S_\ell(B,r) \ge c_\ell(N(B)) \bigr\} \\
\label{eq:Def.U}
	U(x) \
	:= \ &\inf_{B \in \BB\colon \min(B) \ge x} \, Y_{B:N(B)+1-c_u(N(B))} \\
\nonumber
	= \ &\inf_{B \in \BB\colon \min(B) \ge x} \,
		\max \bigl\{ r \in (-\infty,\infty] : S_u(B,r) \ge c_u(N(B)) \bigr\}
\end{align}
with the conventions that $\sup(\emptyset) := -\infty$ and $\inf(\emptyset) := \infty$, where $S_\ell(B,r) := \#\{i\colon X_i \in B, Y_i \le r\}$ and $S_u(B,r) := \#\{i\colon X_i \in B, Y_i \ge r\}$ for $r \in [-\infty,\infty]$.

\paragraph{Finding the critical values $c_\ell(m)$ and $c_u(m)$.}
Let $\tilde{n} \le n$ be the number of elements of the set $\{X_1,\ldots,X_n\}$. Then $\BB$ consists of at most $\tilde{n} (\tilde{n} + 1)/2$ different intervals. By Bonferroni's inequality, for any choice of the bounds $c_\ell(m)$ and $c_u(m)$,
\begin{align}
\nonumber
	\Pr \bigl( & T_\ell(B) < c_\ell(N(B)) \ \text{or} \ T_u(B) < c_u(N(B))
		\ \text{for some} \ B \in \BB \bigr) \\
\label{ineq:Bonferroni}
	&\le \ \sum_{m=1}^n h_m
		\bigl[ F_{m,\gamma}(c_\ell(m)-1) + F_{m,1-\gamma}(c_u(m)-1) \bigr] ,
\end{align}
where $h_m$ is the number of intervals $B \in \BB$ such that $N(B) = m$. Now let
\begin{equation}
\label{eq:UI.crit.values}
	c_\ell(m) \ := \ F_{m,\gamma}^{-1}(\kappa)
	\quad\text{and}\quad
	c_u(m) \ := \ F_{m,1-\gamma}^{-1}(\kappa)
\end{equation}
for some $\kappa \in (0,1]$. With a simple bisection search, one can find a maximal value of $\kappa$ such that the bound \eqref{ineq:Bonferroni} is not larger than $\alpha$. Since $\sum_{m=1}^n H(m) \le \tilde{n} (\tilde{n} + 1)/2$ and $F_{m,p}(F_{m,p}^{-1}(\kappa) - 1) \le \kappa$, we know that
\[
	\kappa \ \ge \ \frac{\alpha}{2\# \BB} \ \ge \ \frac{\alpha}{\tilde{n}(\tilde{n}+1)} .
\]
The simple choice $\kappa = \alpha / (\tilde{n}(\tilde{n}+1))$ would be similar to the approach of \cite{Yang_Barber_2019} or \cite{Dimitriadis_etal_2022}.

Alternatively, we propose to use the critical values in \eqref{eq:UI.crit.values} with $\kappa$ equal to a Monte-Carlo estimate of the $\alpha$-quantile of the random variable
\begin{equation}
\label{eq:UI-P-value}
	\min_{B \in \BB} \
		\min \bigl\{ F_{N(B),\gamma}(T_\ell(B)), F_{N(B),1-\gamma}(T_u(B)) \bigr\} .
\end{equation}

Both proposals correspond to a union-intersection test: For each interval $B \in \BB$ and any candidate $g$ for $Q_\gamma$, one may interpret
\[
	F_{N(B),\gamma}(S_\ell(B,g))
\]
as a p-value of the null hypothesis that $Q_\gamma \le g$ on $B$, and
\[
	F_{N(B),1-\gamma}(S_u(B,g))
\]
is a p-value of the null hypothesis that $Q_\gamma \ge g$ on $B$.

\paragraph{Choosing $\BB$.}
In principle, one could take $\BB$ to be the set of all $\tilde{n}(\tilde{n}+1)/2$ compact intervals $B \subset \XX$ with endpoints in $\{X_1,\ldots,X_n\}$. But for very large values $\tilde{n}$, the Bonferroni version of our critical values $c_\ell(m)$ and $c_u(m)$ may become rather conservative, while Monte-Carlo simulation of \eqref{eq:UI-P-value} may be too time-consuming. Alternatively, we propose to specify a subset $\mathcal{D}$ of $\{1,\ldots,\lceil\tilde{n}/2\rceil\}$ and to restrict ourselves to intervals $B \subset \XX$ such that the cardinality of $B \cap \{x_1,\ldots,x_n\}$ is in $\mathcal{D}$. Specific choices are the intersection of $\{1,\ldots,\lceil\tilde{n}/2\rceil\}$ with the set of numbers $a_\ell := 1 + \ell(\ell-1)/2$, $\ell \ge 1$. That is, $a_1 = 1$ and $a_{\ell+1} - a_\ell = \ell$. Here, $\# \mathcal{D} = O(\tilde{n}^{1/2})$ and thus, $\#\BB = O(\tilde{n}^{3/2})$. Alternatively, one could intersect $\{1,\ldots,\lceil\tilde{n}/2\rceil\}$ with the set of Fibonacci numbers $1,2,3,5,8,13,\ldots$ or the powers $2^\ell$, $\ell \ge 0$. In both cases, $\# \mathcal{D} = O(\log \tilde{n})$ and $\#\BB = O(\tilde{n} \log \tilde{n})$.

\paragraph{Explicit computation of $L$ and $U$.}
Let $z_1 < \cdots < z_{\tilde{n}}$ be the elements of $\{X_1,\ldots,X_n\}$. Note first that $L \equiv -\infty$ on $(-\infty,z_1)$, and $L$ is constant on all intervals $[z_j,z_{j+1})$, $1 \le j \le \tilde{n}$, where $z_{\tilde{n}+1} := \infty$. Similarly, $U \equiv \infty$ on $(z_{\tilde{n}},\infty)$, and $U$ is constant on all intervals $(z_{j-1},z_j]$, $1 \le j \le \tilde{n}$, where $z_0 := -\infty$. Thus it suffices to compute $\bs{L} = (L(z_j))_{j=1}^{\tilde{n}}$ and $\bs{U} = (U(z_j)_{j=1}^{\tilde{n}}$.

Now we describe an algorithm to compute $\bs{L}$ in $O(n^2)$ steps. The same algorithm can be applied to compute $\bs{U} = (U(z_j))_{j=1}^{\tilde{n}}$ by replacing temporarily the observations $(X_i,Y_i)$ with $(-X_i,-Y_i)$ and $\gamma$ with $1 - \gamma$. This is possible because our methods of setting $\kappa$ for the critical values in \eqref{eq:UI.crit.values} are invariant under these replacements.

In what follows, let $N_j := \#\{i\colon X_i \le z_j\}$ for $j \in \{0,1,\ldots,\tilde{n}\}$. Suppose that for some $k \in \{1,\ldots,\tilde{n}\}$, the value $L(z_{k-1})$ is known already, where $L(z_0) := -\infty$. Then $L(z_k)$ is the smallest number $r \in [L(z_{k-1}), \infty)$ such that
\[
	S([z_j,z_k],r) \ \ge \ c_\ell(N_k - N_{j-1})
	\quad\text{for all} \ [z_j,z_k] \in \BB .
\]
To determine this number $r$, we start with $r = L(z_{k-1})$ and check for $j = k, k-1, k-2, \ldots$ whether $[z_j,z_k]$ belongs to $\BB$ and, if yes, satisfies $S([z_j,z_k],r) \ge c_\ell(N_k - N_{j-1})$. While doing this, we determine the next candidate
\[
	r_{\rm new} \ := \ \min \bigl( \{\infty\} \cup \{Y_i \colon X_i \in [z_j,z_k], Y_i > r \bigr)
\]
for $r$. As soon as $[z_j,z_k] \in \BB$ but $S([z_j,z_k],r) < c_\ell(N_k - N_{j-1})$, we replace $r$ with $r_{\rm new}$ and reset $j$ to $k$. This leads to the algorithm described in Table~\ref{tab:AlgorithmLB}.

\begin{table}
\[
	\begin{array}{|l|}
	\hline
	r \leftarrow -\infty \\
	\text{for} \ k = 1, 2, \ldots, n \ \text{do}\\
	\strut\quad j \leftarrow k+1 \\
	\strut\quad S \leftarrow 0 \\
	\strut\quad r_{\rm new} \leftarrow \infty \\
	\strut\quad \text{while} \ j > 1 \ \text{do} \\
	\strut\quad\quad j \leftarrow j-1 \\
	\strut\quad\quad S \leftarrow
		S + \sum_{i\colon X_i = z_j} 1_{[Y_i \le r]} \\
	\strut\quad\quad r_{\rm new} \leftarrow
		\min \bigl( \{r_{\rm new}\} \cup \{Y_i\colon X_i = z_j, Y_i > r\} \bigr) \\
	\strut\quad\quad \text{if} \ S < c_\ell(N_k - N_{j-1}) \ \text{then}\\
	\strut\quad\quad\quad r \leftarrow r_{\rm new} \\
	\strut\quad\quad\quad j \leftarrow k+1 \\
	\strut\quad\quad\quad S \leftarrow 0 \\
	\strut\quad\quad\quad r_{\rm new} \leftarrow \infty \\
	\strut\quad\quad \text{end if} \\
	\strut\quad \text{end while} \\
	\strut\quad L(z_k) \leftarrow r \\
	\text{end for} \\
	\hline
	\end{array}
\]
\caption{Algorithm to compute the lower confidence bounds $L(\cdot)$.}
\label{tab:AlgorithmLB}
\end{table}

That this algorithm has running time $O(n^2)$ can be verified as follows: For a given pair $(k,r)$ in $\{1,\ldots,\tilde{n}\} \times (\{-\infty\} \cup \{Y_1,\ldots,Y_n\})$, we initialise and update $S$ and $r_{\rm new}$ for $j = k, k-1, k-2, \ldots$ until $j = 1$ or $[z_j,z_k] \in \BB$ but $S([z_j,z_k]) < c_\ell(N_k - N_{j-1})$. This requires $O(n)$ steps. The next pair $(k',r')$ satisfies either $k' > k$ or $r' > r$. Consequently, from the starting point $(k,r) = (1,-\infty)$, we arrive at the point $(k,r) = (\tilde{n}, L(z_{\tilde{n}}))$ in at most $\tilde{n}-1 + n = O(n)$ steps. Hence, the total running time is of order $O(n^2)$.

\section{Asymptotics}
\label{sec:Asymptotics}

Similarly as in \cite{Duembgen_Johns_2004}, \cite{Moesching_Duembgen_2020} and \cite{Dimitriadis_etal_2022}, we consider a triangular scheme of observations. For each sample size $n \ge 1$, we observe $(X_{n1}, Y_{n1}), \ldots, (X_{nn}, Y_{nn})$ with fixed numbers $X_{n1} \le \ldots \le X_{nn}$ in $\XX$ and independent random variables $Y_{n1}, \ldots, Y_{nn}$, where $Y_{ni}$ has distribution function $F(\cdot \,|\, X_{ni})$. Let $(L_n,U_n)$ be the confidence band described before for sample size $n$, based on a family $\BB_n \subset \BB_n^*$, where $\BB_n^*$ is the family of all compact intervals with endpoints in $\{X_{n1},\ldots,X_{nn}\}$. We analyze asymptotic properties of $(L_n,U_n)$ on a bounded open interval $(a_0,b_o) \subset \XX$ under the following three assumptions.

\textbf{(A)} For any set $B_* \in \BB_n^*$, there exists a set $B \in \BB_n$ such that $B \subset B^*$ and $N_n(B) \ge N_n(B^*)/2$.

\textbf{(B)} Let $\mathrm{Leb}(\cdot)$ denote Lebesgue measure, and let $N_n(B) = \#\{i\colon X_{ni} \in B\}$ for $B \subset \XX$. There exist constants $C_1, C_2 > 0$ such that for sufficiently large $n$,
\[
	N_n(B) \ \ge \ C_1 n \, \mathrm{Leb}(B)
\]
for arbitrary intervals $B \subset (a_o,b_o)$ such that $\mathrm{Leb}(B) \ge C_2 \log(n)/n$.

\textbf{(C)} There exist constants $D_1, D_2 > 0$ such that for all $x \in (a_o,b_o)$, the $\gamma$-quantile $Q_\gamma(x)$ is unique, and
\[
	\frac{F(Q_\gamma(x) + h \,|\, x) - \gamma}{h} \ \ge \ D_1
\]
for all nonzero $h \in [-D_2,D_2]$.

One can easily show that assumption~(A) is satisfied for all proposals of $\BB_n$ in the previous section. Assumption~(B) is satisfied, for instance, in a deterministic setting with fixed design points $X_{ni} = G^{-1}(i/n)$ for some distribution function $G$ on $\XX$ such that $G'$ exists and is bounded away from $0$ on $(a_o,b_o)$. If the $X_{ni}$ are the order statistics of $n$ independent random variables with such a distribution function $G$, the arguments in Section~4.3 of \cite{Moesching_Duembgen_2020} can be modified to show that Condition~(A) is satisfied almost surely, provided that $C_1, C_2 > 0$ are chosen appropriately.

\begin{Theorem}
\label{thm:asymptotics}
Suppose that assumptions~(A,B,C) are satisfied. Let $\rho_n = \log(n)/n$. There exist constants $C, C' > 0$ depending only on $C_1,C_2,D_1,D_2$ with the following properties:

\noindent
(i) Suppose that $Q_\gamma$ is constant on $(a_o,b_o)$. Then with asymptotic probability one,
\begin{align*}
	U_n(x) \ &\le \ p(x) + C \sqrt{\rho_n/(b_o - x)}
		\quad\text{for all} \ x \in (a_o,b_o - C'\rho_n] , \\
	L_n(x) \ &\ge \ p(x) - C \sqrt{\rho_n/(b_o - x)}
		\quad\text{for all} \ x \in [a_o+C'\rho_n,b_o) .
\end{align*}

\noindent
(ii) Suppose that $p$ is H{\"o}lder-continuous with exponent $\beta \in (0,1]$ and constant $\lambda > 0$ on $(a_o,b_o)$, that is, $\bigl| Q_\gamma(x) - Q_\gamma(x') \bigr| \le \lambda |x - x'|$ for all $x,x' \in (a_o,b_o)$. Then for $\delta_n := (\lambda^{-2} \rho_n)^{1/(2\beta + 1)}$,
\[
	\sup_{x \in (a_o,b_o-\delta_n]} \,
		\bigl( U_n(x) - Q_\gamma(x) \bigr)^+ , \,
	\sup_{x \in (a_o + \delta_n, b_o)} \,
		\bigl( Q_\gamma(x) - L_n(x) \bigr)^+ \
	\le \ C \lambda^{1/(2\beta + 1)} \rho_n^{\beta/(2\beta + 1)} .
\]
	
\noindent
(iii) Suppose that $Q_\gamma$ is discontinuous at some point $x_o \in (a_o,b_o)$. Then with asymptotic probability one,
\begin{align*}
	U_n(x) \ &\le \ Q_\gamma(x_o-) + C \sqrt{\rho_n/(x_o - x)}
		\quad\text{for all} \ x \in (a_o,x_o - C'\rho_n] , \\
	L_n(x) \ &\ge \ Q_\gamma(x_o+) - C \sqrt{\rho_n/(x - x_o)}
		\quad\text{for all} \ x \in [a_o+C'\rho_n,b_o) .
\end{align*}
\end{Theorem}

Part~(i) shows that for arbitrary fixed numbers $a_o < a < b < b_o$,
\[
	\sup_{x \in (a_o,b]} \, (U_n - p)^+(x)
	+ \sup_{x \in [a,b_o)} \, (p - L_n)^+(x) \
	= \ O_p(\rho_n^{1/2})
\]
Part~(iii) shows that for with asymptotic probability one,
\[
	U_n(x_\ell) \
	< \ \frac{Q_\gamma(x_o-) + Q_\gamma(x_o+)}{2} \
	< \ L_n(x_r)
\]
for $a_o < x_\ell < x_o - D\rho_n$ and $x_o + D\rho_n < x_r < b_o$, where $D$ is the maximum of $C'$ and $4 C^2 \bigl( Q_\gamma(x_o+) - Q_\gamma(x_o-) \bigr)^{-2}$. Consequently, whenever $Q_\gamma$ has a discontinuity at some point in $(a_o,b_o)$, the confidence band $(L_n,U_n)$ crosses a horizontal line on an interval of length $O(\rho_n)$.

\begin{proof}[\bf Proof of Theorem~\ref{thm:asymptotics}]
For symmetry reasons, it suffices to prove the claims about the lower bound function $L_n(\cdot)$. For notational convenience, the additional subscript $n$ for the triangular scheme is often dropped. In what follows, $C_o$ denotes a generic universal constant, and $C'$ denotes a generic constant depending possibly on $C_1, D_1$.

Applying Hoeffding's inequality \citep{Hoeffding_1963} to binomial distributions leads to the inequalities
\begin{equation}
\label{ineq:Hoeffding1}
	c_\ell(m) \ \ge \ m\gamma - \sqrt{m \log(1/\kappa)/2} ,
\end{equation}
where $c_\ell(m) = F_{m,\gamma}^{-1}(\kappa)$. Since $\kappa \ge \alpha/[n(n+1)]$, we may conclude from \eqref{ineq:Hoeffding2} that for $1 \le m \le n$ and sufficiently large $n$,
\begin{equation}
\label{ineq:Hoeffding2}
	c_\ell(m) \ \ge \ m\gamma - C_o \sqrt{m \log(n)} ,
\end{equation}
provided that $C_o > 1$. For $B \in \BB$ and $r \in \R$, let
\[
	\Sigma_\ell(B,r) \ := \ \Ex S_\ell(B,r) .
\]
Then it follows from \citet[Corollary~4.7]{Moesching_Duembgen_2020} that with asymptotic probability one,
\begin{equation}
\label{ineq:Hoeffding3}
	\sup_{B \in \BB, r \in \R} \, N(B)^{-1/2}
		\bigl| (S_\ell - \Sigma_\ell)(B,r) \bigr| \
	\le \ C_o \sqrt{\log (n)} ,
\end{equation}
provided that $C_o > 1$.

From now on, we always assume that $n$ is large enough such that all inequalities for $N(\cdot)$ in assumption~(A) are satisfied, and we tacitly assume that inequalities \eqref{ineq:Hoeffding2} and \eqref{ineq:Hoeffding3} hold true as well. For $B \in \BB$ and $j \in \{0,1,\ldots,N(B)\}$, it follows from \eqref{ineq:Hoeffding3} that
\[
	j \ \le \ S_\ell(B,Y_{B:j}) \ \le \ \Sigma_\ell(B,Y_{B:j}) + C_o \sqrt{N(B) \log(n)} .
\]
Plugging in $j = c_\ell(N(B))$ and using \eqref{ineq:Hoeffding2}, we may conclude that for all $B \in \BB$,
\[
	\Sigma_\ell(B,Y_{B:c_\ell(N(B))}) \
	\ge \ N(B) \gamma - C_o \sqrt{N(B) \log(n)} ,
\]
provided that $C_o > 2$. Combining this inequality with the definition~\eqref{eq:Def.L} of $L$ shows that $L$ is bounded from below by the following deterministic function $\tilde{L}$:
\[
	\tilde{L}(x) \
	:= \ \sup_{B\in\BB\colon \max(B) \le x} \,
		\inf \bigl\{r \in \R \colon
			\Sigma_\ell(B,r) \ge N(B) \gamma - C_o \sqrt{N(B) \log(n)} \bigr\} .
\]
Consequently, it suffices to verify the assertions for $\tilde{L}$ instead of $L$.

As to part~(i), suppose that $Q_\gamma \equiv q$ on $(a_o,b_o)$. For $x \in [a_o+C_2\rho_n, b_o)$, let $B(x)$ be the largest interval in $\BB_n$ which is contained in $(a_o,x]$. By assumptions~(A-B), $N(B(x)) \ge 2^{-1} C_1 n (x - a_o)$, and assumption~(C) implies that for $h \in (0,D_2]$,
\[
	\Sigma_\ell(B(x),q - h) \
	= \ \sum_{i\colon X_i \in B(x)} F(Q_\gamma(X_i) - h \,|\, X_i) \
	\le \ N(B(x)) \gamma - D_1 N(B(x)) h .
\]
Consequently, $\Sigma_\ell(B(x)) \ge N(B(x)) \gamma - C_o \sqrt{N(B(x)) \log(n)}$ implies that
\[
	h \
	\le \ C_o D_1^{-1} \sqrt{\log(n) / N(B(x))} \
	\le \ h_n(x) := 2^{1/2} C_o D_1^{-1} C_1^{-1/2} \sqrt{\rho_n/(x - a_o)} .
\]
If we restrict $x$ to $[a_0+C'\rho_n,b_o)$ with sufficiently large $C' \ge C_2$, the number $h_n(x)$ is no larger than $D_2$, and we may conclude that
\[
	\tilde{L}(x) \
	\ge \ \inf \bigl\{ r \in \R \colon
		\Sigma_\ell(B(x),r) \ge N(B(x)) \gamma - C_o \sqrt{N(B(x)) \log(n)} \bigr\} \
	\ge \ q - h_n(x) ,
\]
which proves part~(i).

Concerning part~(ii), suppose that $Q_\gamma$ is H\"{o}lder-continuous on $(a_o,b_o)$ with exponent $\beta \in (0,1]$ and constant $\lambda > 0$. Since $\delta_n/\rho_n \to \infty$, we may assume that $\delta_n \ge C_2 \rho_n$. For $x \in [a_o+\delta_n,b_o)$, let $B(x)$ be the largest interval in $\BB_n$ which is contained in $(x - \delta_n, x]$. Then assumptions (A-B) imply that
$N(B(x)) \ge 2^{-1} C_1 n \delta_n$. For $h \in (0,D_2]$, it follows from assumption~(C) that
\begin{align*}
	\Sigma_\ell(B(x),Q_\gamma(x) - h) \
	&= \ \sum_{i\colon X_i \in B(x)} F(Q_\gamma(X_i) - h \,|\, X_i) \\
	&\le \ \sum_{i\colon X_i \in B(x)} F(Q_\gamma(x) + L \delta_n^\beta - h \,|\, X_i) \\
	&\le \ N(B(x)) - D_1 N(B(x)) (h - L\delta_n^\beta)^+ .
\end{align*}
Consequently, $\Sigma_\ell(B(x),Q_\gamma(x) - h) \ge N(B(x)) \gamma - C_o \sqrt{N(B(x)) \log(n)}$ implies that
\[
	h \
	\le \ L \delta_n^\beta + C_o \sqrt{N(B(x)) \log(n)} \
	\le h_n ,
\]
where
\[
	h_n \
	:= \ L \delta_n^\beta + 2^{1/2} C_o C_1^{-1/2} \sqrt{\rho_n/\delta_n} \
	= \ (1 + 2^{1/2} C_o C_1^{-1/2}) \lambda^{1/(2\beta+1)} \rho_n^{\beta/(2\beta+1)} .
\]
Since $h_n \to 0$, the requirement $h_n \le D_2$ is satisfied for sufficiently large $n$, and then for all $x \in [a_o+\delta_n,b_o)$,
\[
	\tilde{L}(x) \
	\ge \ \inf \bigl\{ r \in \R \colon
		\Sigma_\ell(B(x),r) \ge N(B(x)) \gamma - C_o \sqrt{N(B(x)) \log(n)} \bigr\} \
	\ge \ Q_\gamma(x) - h_n ,
\]
which proves part~(ii).

Finally, suppose that $Q_\gamma$ is discontinuous at $x_o \in (a_o,b_o)$. For $x \in (a_o,x_o - C_2\rho_n]$, let $B(x)$ be the largest interval in $\BB_n$ which is contained in $[x,x_o)$. By assumptions~(A-B), $N(B(x)) \ge 2^{-1} C_1 n (x_o - x)$. For $h \in (0,D_2]$, it follows from assumption~(C) that
\[
	\Sigma_\ell(B(x),Q_\gamma(x_o-) - h) \
	\le \ \sum_{i\colon X_i \in B(x)} F(Q_\gamma(X_i) - h \,|\, X_i) \
	\le \ N(B(x)) \gamma - D_1 N(B(x)) h .
\]
Consequently, $\Sigma_\ell(B(x)) \ge N(B(x)) \gamma - C_o \sqrt{N(B(x)) \log(n)}$ implies that
\[
	h \
	\le \ C_o D_1^{-1} \sqrt{\log(n) / N(B(x))} \
	\le \ h_n(x) := 2^{1/2} C_o D_1^{-1} C_1^{-1/2} \sqrt{\rho_n/(x_o - x)} .
\]
If we restrict $x$ to $(a_o,x_o - C'\rho_n]$ for sufficiently large $C' \ge C_2$, then $h_n(x) \le D_2$, and we may conclude that
\[
	\tilde{L}(x) \
	\ge \ \inf \bigl\{ r \in \R \colon
		\Sigma_\ell(B(x),r) \ge N(B(x)) \gamma - C_o \sqrt{N(B(x)) \log(n)} \bigr\} \
	\ge \ q - h_n(x) ,
\]
which proves part~(iii).
\end{proof}

\section{Numerical examples}
\label{sec:Examples}

Figure~\ref{fig:Data_N500} shows a simulated data set of $n = 500$ pairs $(X_i,Y_i)$, $1 \le i \le n$, where $X_i \in (0,50)$. To computed $95\%$-confidence bands for a corresponding quantile function $Q_\gamma$, we used the family $\BB$ of all compact intervals $B$ with endpoints in $\{X_1,\ldots,X_n\}$ such that $\# B$ is in $\{1 + \ell(\ell-1)/2 \colon \ell \ge 1\}$.

For $\gamma = 0.5$ and $\alpha = 0.05$, the Bonferroni bound for the critical value $\kappa$ was $1.041119 \cdot 10^{-5}$, but by means of $199999$ Monte Carlo simulations of \eqref{eq:UI-P-value} we obtained the more accurate and substantially larger value $\hat{\kappa} = 9.901973 \cdot 10^{-4}$. Figure~\ref{fig:CB50_N500} shows the data with the resulting $95\%$-confidence band for $Q_{0.5}$; the latter is shown as well.

For $\gamma \in \{0.25,0.75\}$ and $\alpha = 0.05$, the Bonferroni bound for the critical value $\kappa$ was $1.152103 \cdot 10^{-5}$, but with $199999$ Monte Carlo simulations of \eqref{eq:UI-P-value} we obtained the value $\hat{\kappa} = 1.07446 \cdot 10^{-4}$. Figures~\ref{fig:CB25_N500} and \ref{fig:CB75_N500} show the data with the resulting $95\%$-confidence band for $Q_{0.25}$ and $Q_{0.75}$, respectively.

Figures~\ref{fig:Data_N2500}, \ref{fig:CB50_N2500}, \ref{fig:CB25_N2500} and \ref{fig:CB75_N2500} show analogous pictures for a data set of size $n = 2500$.

\begin{figure}
\centering
\includegraphics[width=0.9\textwidth]{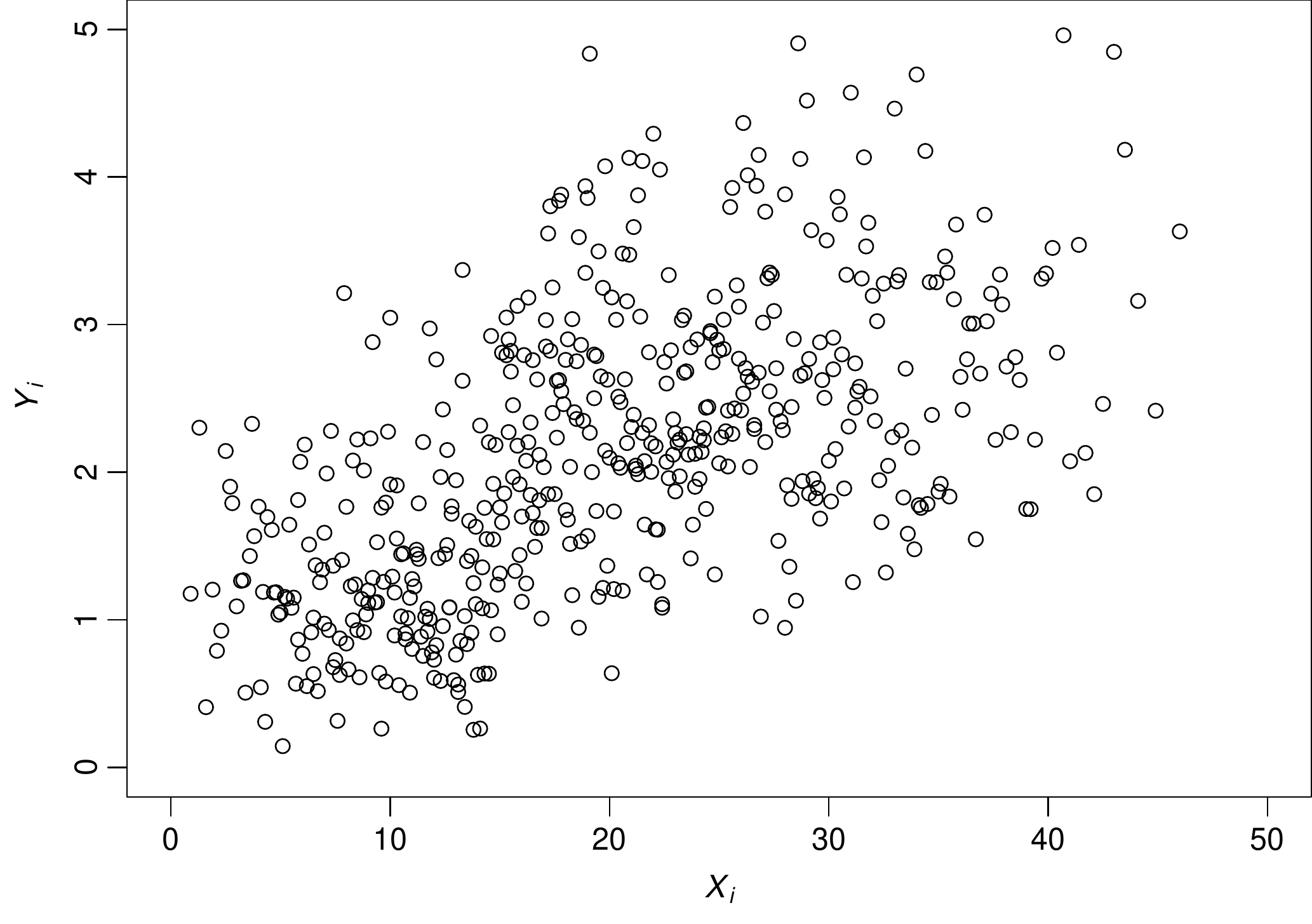}
\caption{An example data set of size $n = 500$.}
\label{fig:Data_N500}
\end{figure}

\begin{figure}
\centering
\includegraphics[width=0.9\textwidth]{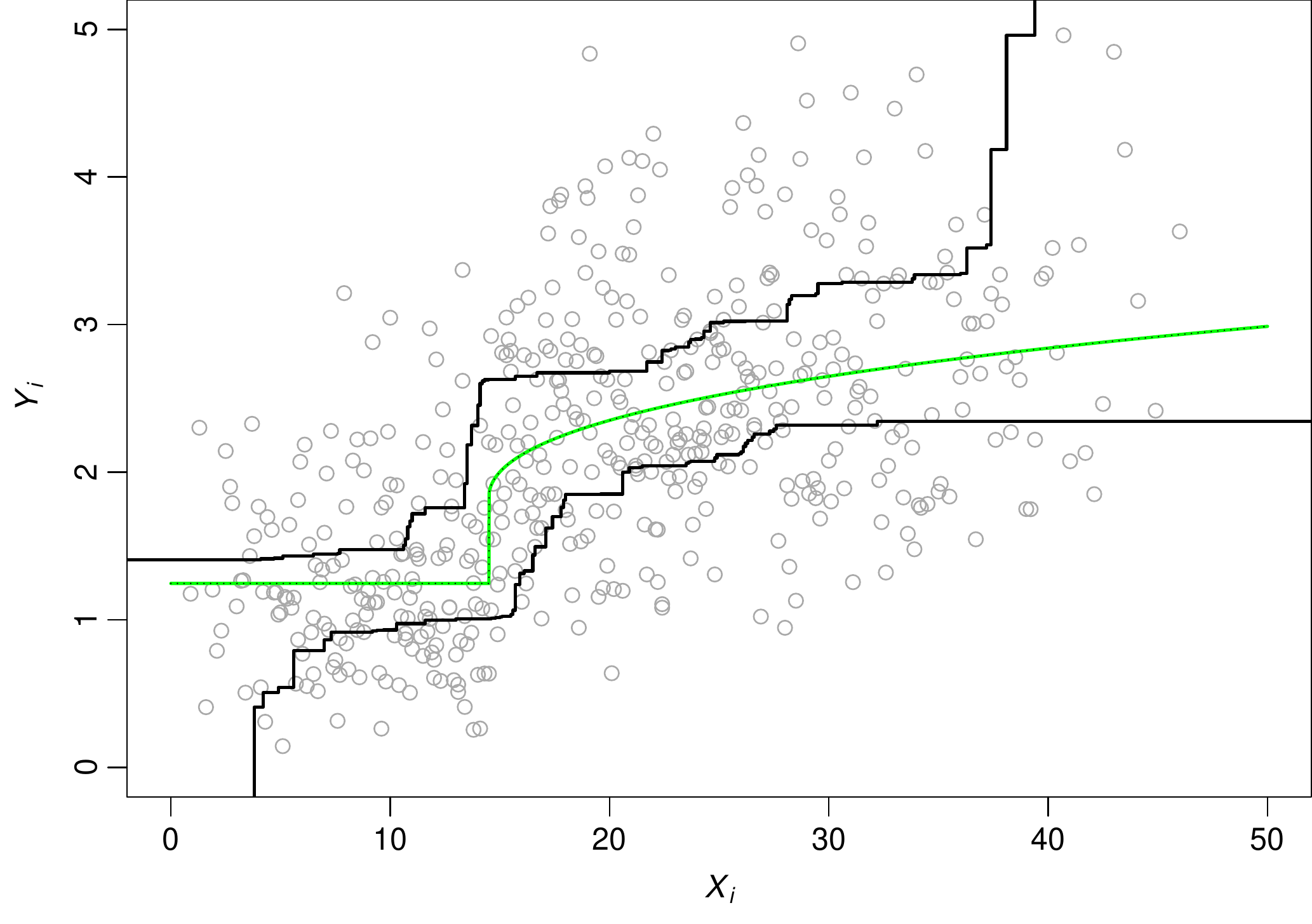}
\caption{$95\%$-Confidence band for $Q_{0.5}$ for the data in Figure~\ref{fig:Data_N500}, together with the true quantile function $Q_{0.5}$ (green, dotted).}
\label{fig:CB50_N500}
\end{figure}

\begin{figure}
\centering
\includegraphics[width=0.9\textwidth]{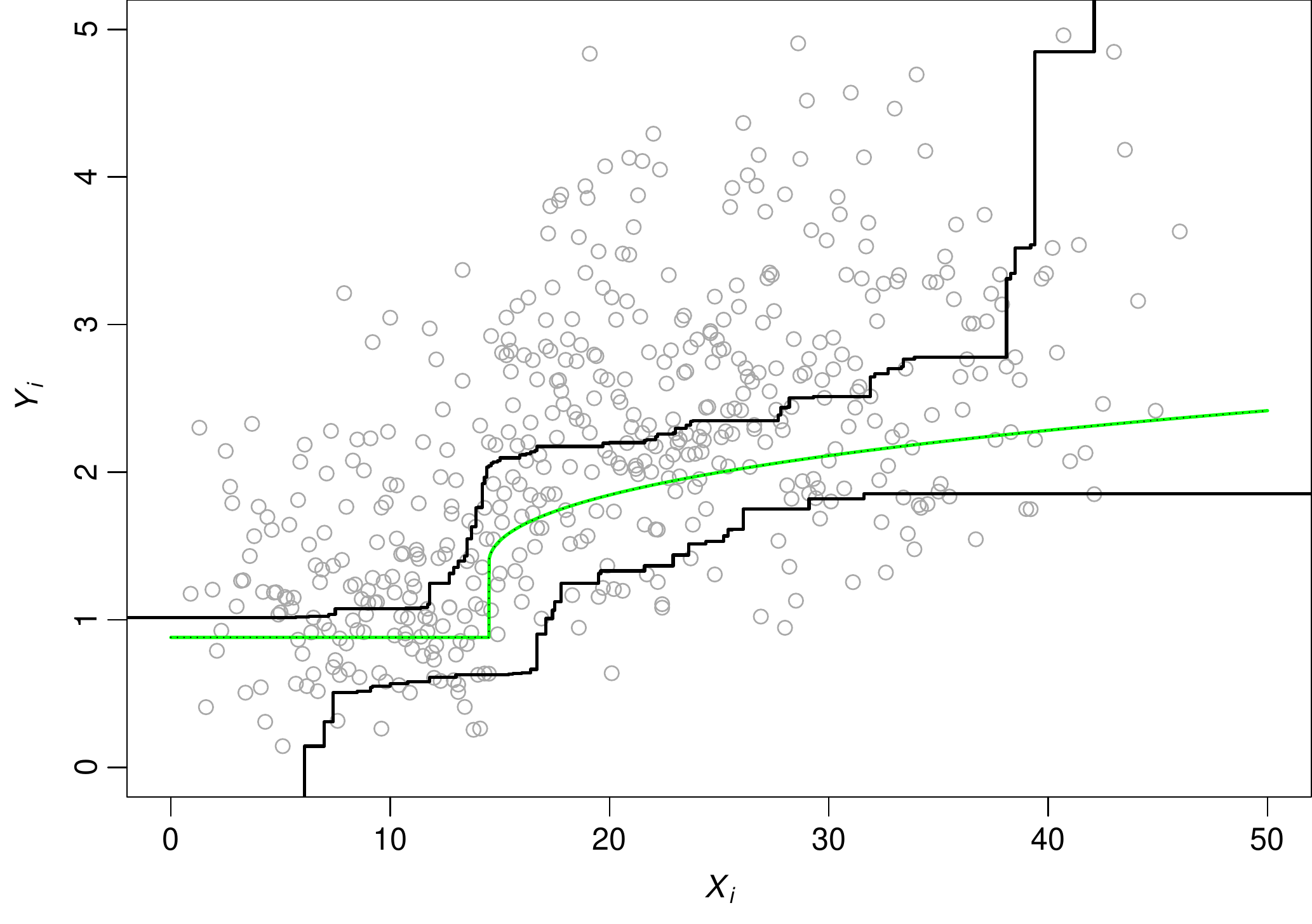}
\caption{$95\%$-Confidence band for $Q_{0.25}$ for the data in Figure~\ref{fig:Data_N500}, together with the true quantile function $Q_{0.25}$ (green, dotted).}
\label{fig:CB25_N500}
\end{figure}

\begin{figure}
\centering
\includegraphics[width=0.9\textwidth]{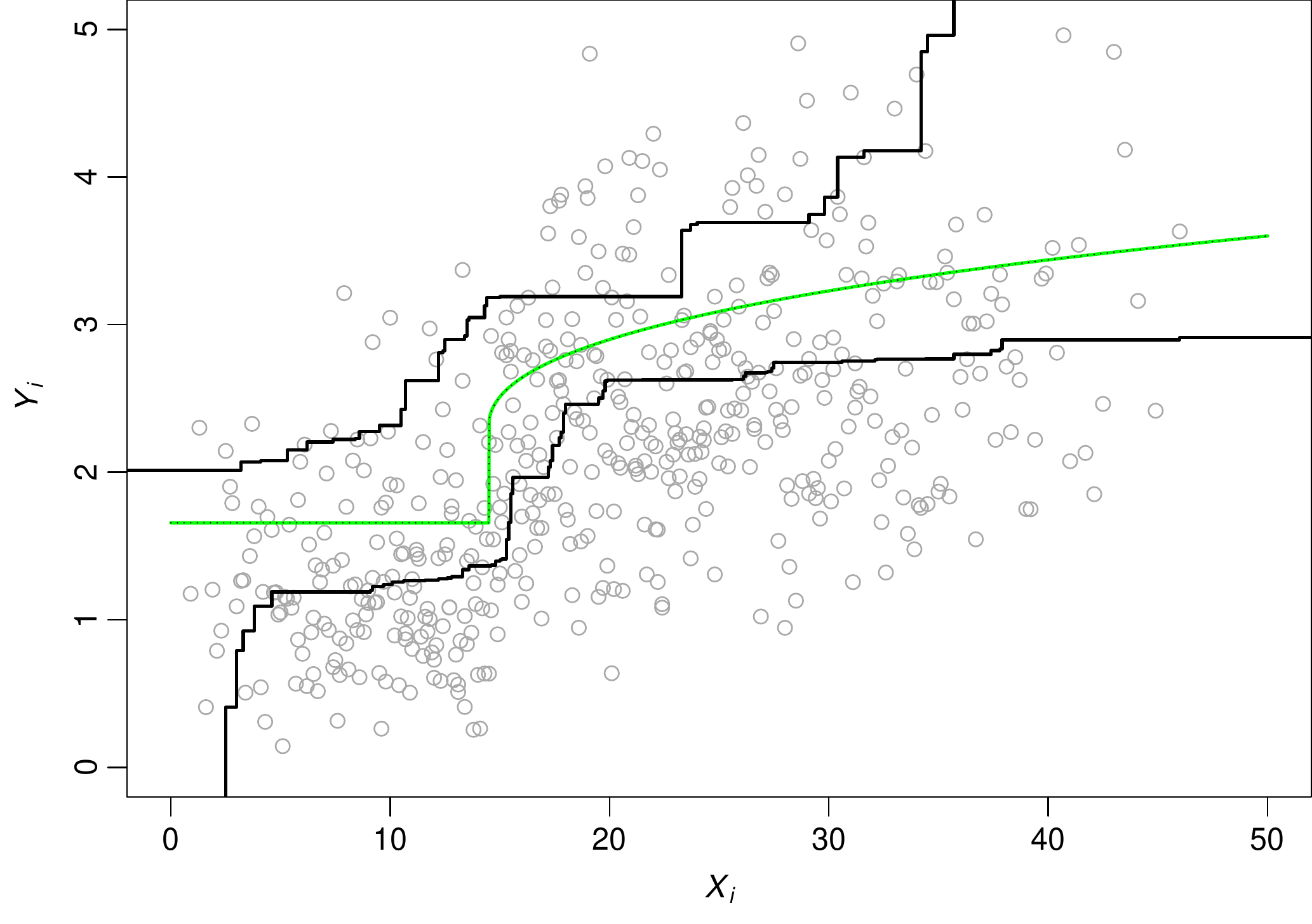}
\caption{$95\%$-Confidence band for $Q_{0.75}$ for the data in Figure~\ref{fig:Data_N500}, together with the true quantile function $Q_{0.75}$ (green, dotted).}
\label{fig:CB75_N500}
\end{figure}

\begin{figure}
\centering
\includegraphics[width=0.9\textwidth]{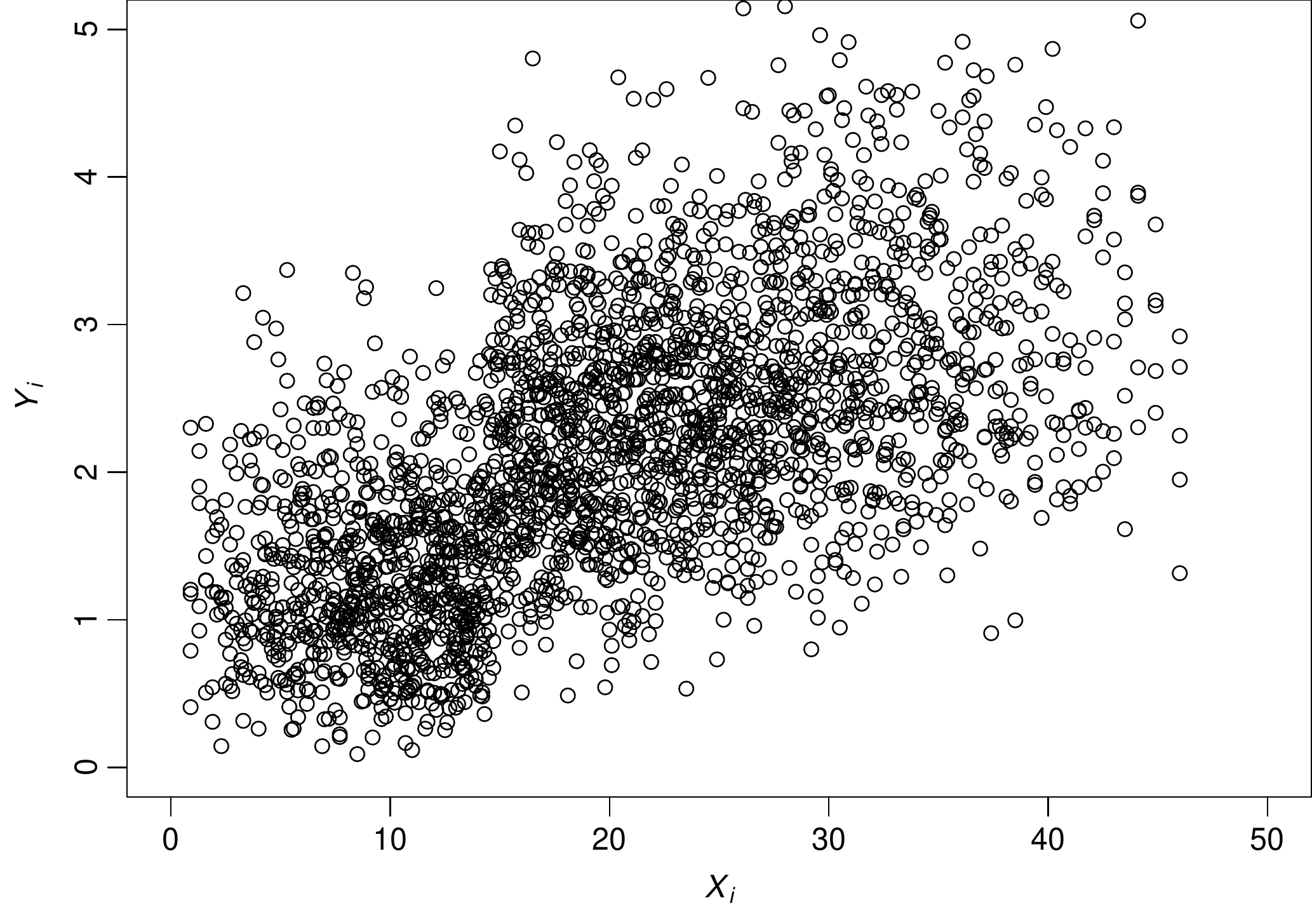}
\caption{An example data set of size $n = 2500$.}
\label{fig:Data_N2500}
\end{figure}

\begin{figure}
\centering
\includegraphics[width=0.9\textwidth]{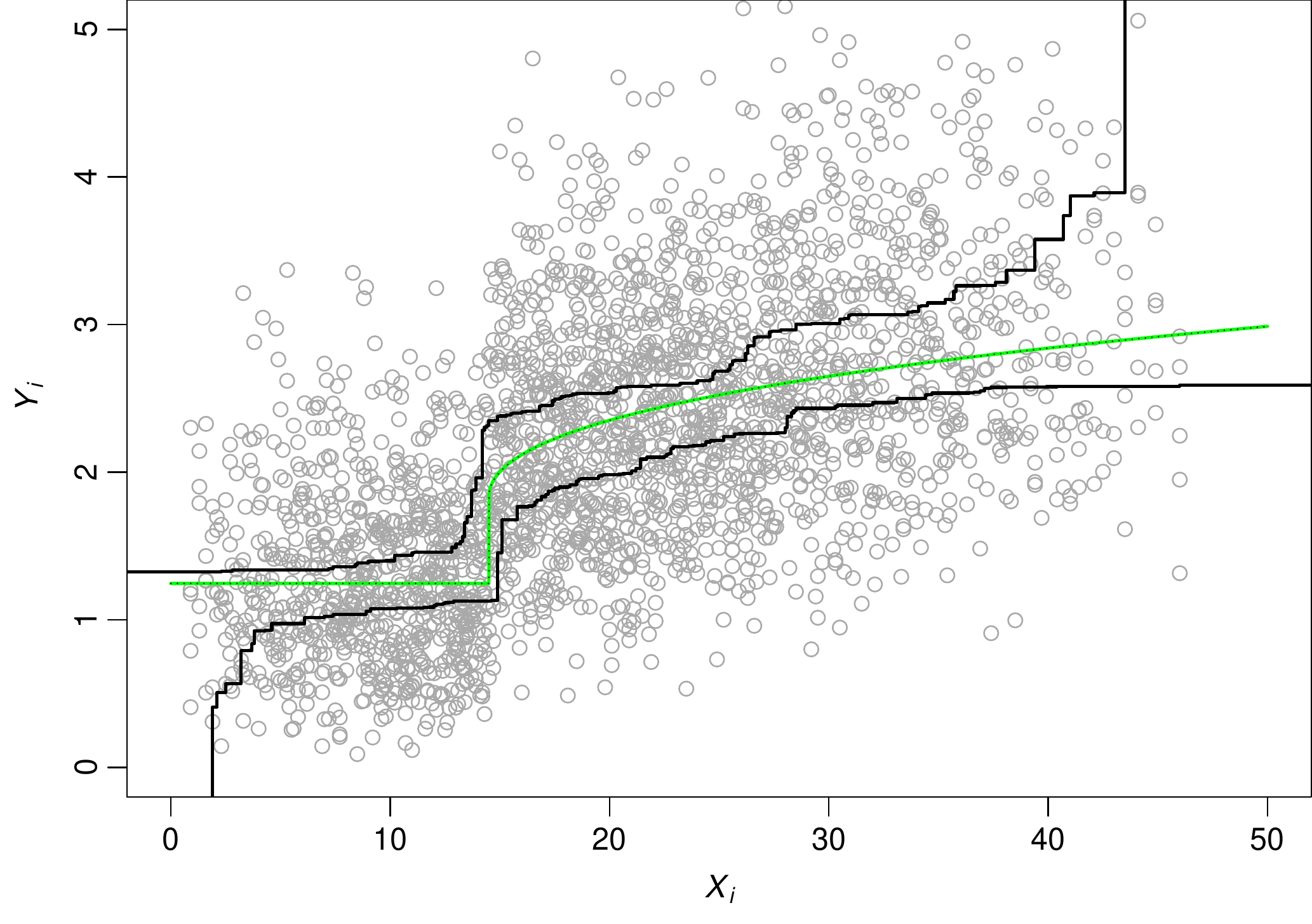}
\caption{$95\%$-Confidence band for $Q_{0.5}$ for the data in Figure~\ref{fig:Data_N2500}, together with the true quantile function $Q_{0.5}$ (green, dotted).}
\label{fig:CB50_N2500}
\end{figure}

\begin{figure}
\centering
\includegraphics[width=0.9\textwidth]{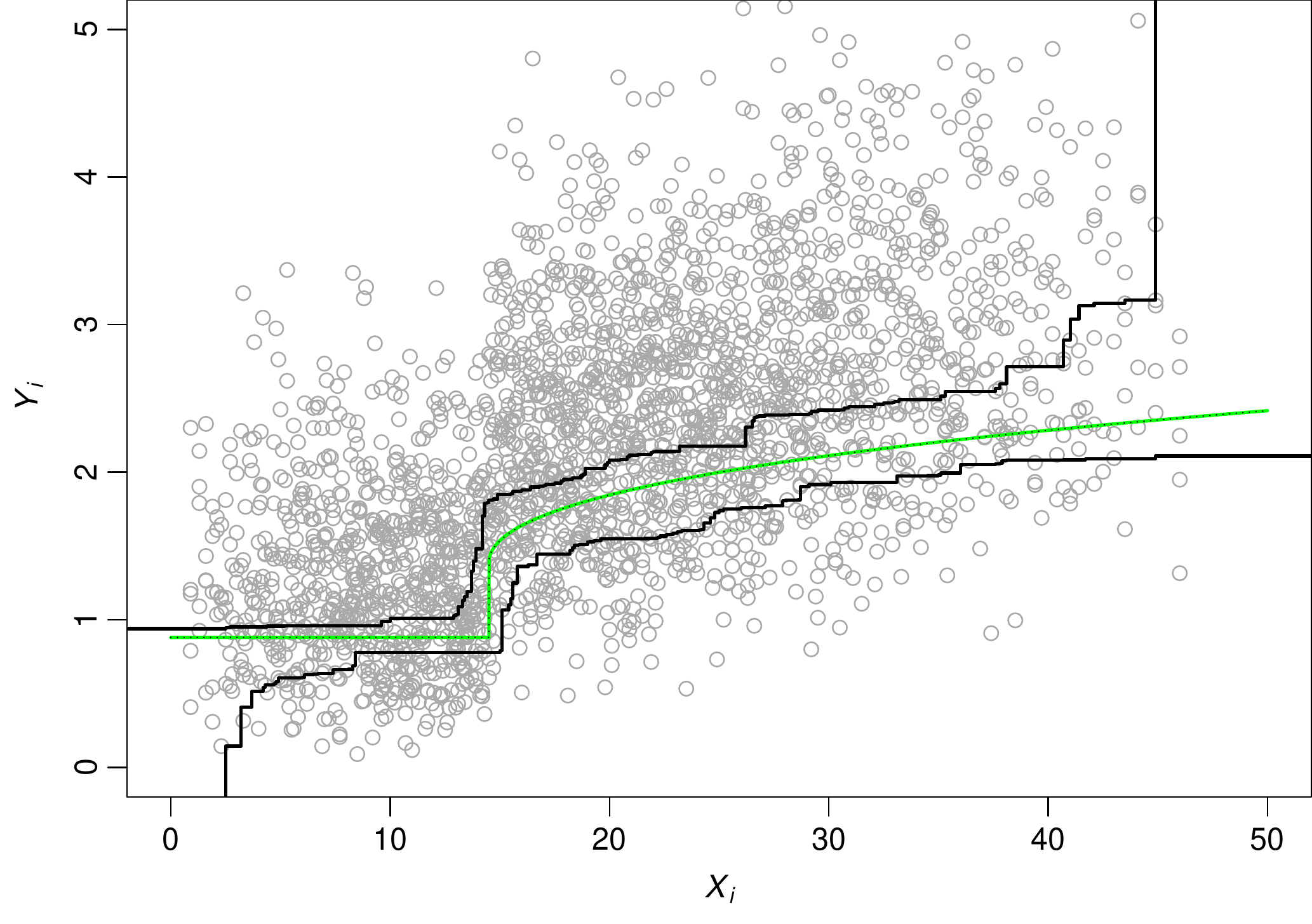}
\caption{$95\%$-Confidence band for $Q_{0.25}$ for the data in Figure~\ref{fig:Data_N2500}, together with the true quantile function $Q_{0.25}$ (green, dotted).}
\label{fig:CB25_N2500}
\end{figure}

\begin{figure}
\centering
\includegraphics[width=0.9\textwidth]{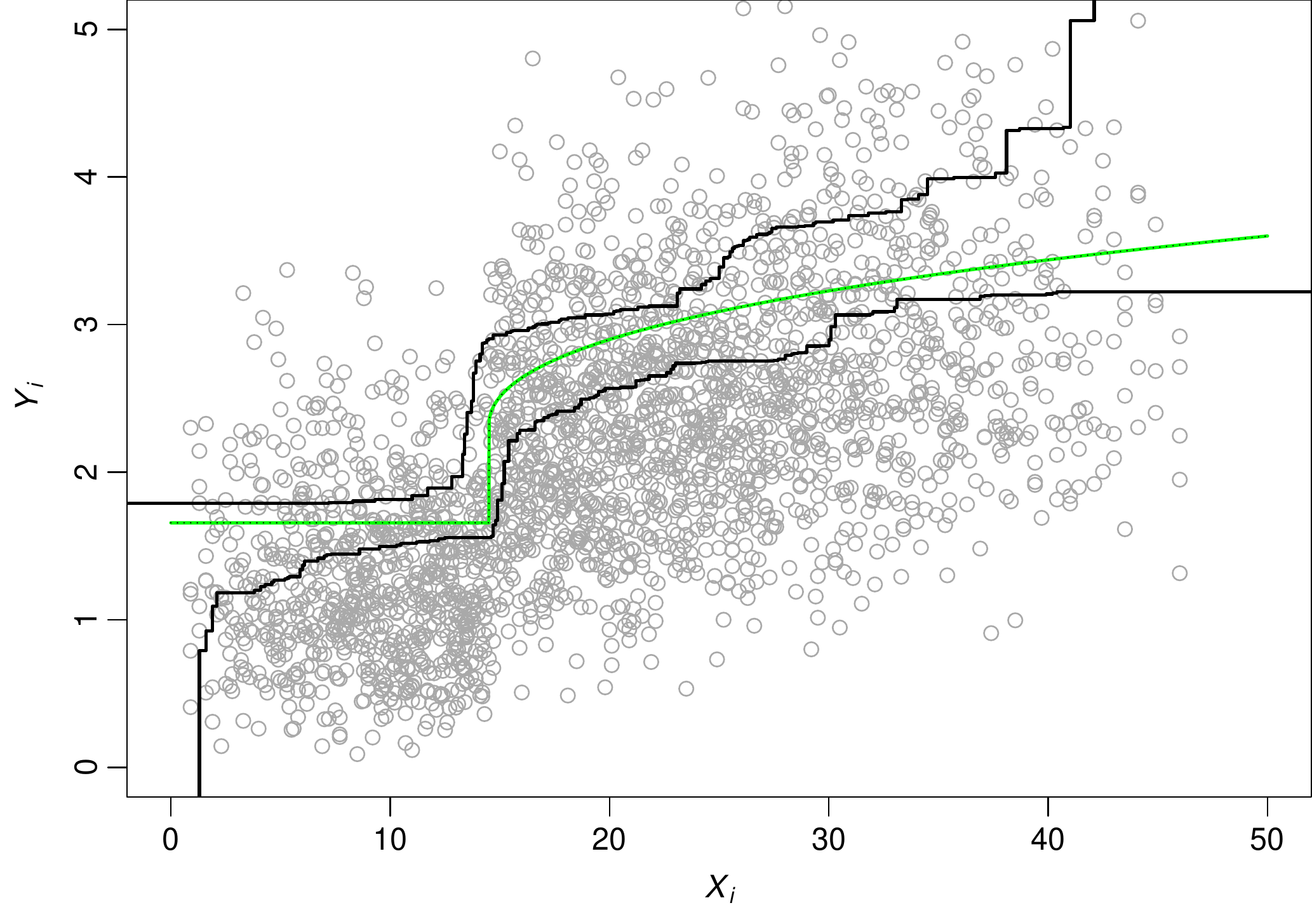}
\caption{$95\%$-Confidence band for $Q_{0.75}$ for the data in Figure~\ref{fig:Data_N2500}, together with the true quantile function $Q_{0.75}$ (green, dotted).}
\label{fig:CB75_N2500}
\end{figure}

\paragraph{An stronger shape constraint.}
In some applications, it is plausible that the quantile functions should be isotonic and convex or isotonic and concave. Both constraints are special cases of the more general assumption that the quantile functions are S-shaped. That means, they are isotonic on $\XX$, and for some inflection point $\mu \in [-\infty,\infty]$, they are convex on $\XX \cap (-\infty,\mu]$ and concave on $[\mu,\infty) \cap \XX$. Since S-shapedness implies isotonicity, assuming the quantile function $Q_\gamma$ is S-shaped allows for the following refinement of our $(1 - \alpha)$-confidence band $(L,U)$ for $Q_\gamma$. We compute $(\tilde{L}, \tilde{U})$ with
\begin{align*}
	\tilde{L}(x) \
	&:= \ \inf \bigl\{ S(x) \colon S \ \text{S-shaped}, L \le S \le U \bigr\} , \\
	\tilde{U}(x) \
	&:= \ \sup \bigl\{ S(x) \colon S \ \text{S-shaped}, L \le S \le U \bigr\} .
\end{align*}
Precisely, we choose a fine grid $M$ of potential values for the inflection point $\mu$ of $Q_\gamma$ and determine
\begin{align*}
	\tilde{L}(x,\mu) \
	&:= \ \inf \bigl\{ S(x) \colon S \ \text{S-shaped with inflection point} \ \mu,
		\, L \le S \le U \bigr\} , \\
	\tilde{U}(x,\mu) \
	&:= \ \sup \bigl\{ S(x) \colon S \ \text{S-shaped with inflection point} \ \mu,
		\, L \le S \le U \bigr\} .
\end{align*}
These functions $\tilde{L}(\cdot,\mu)$ and $\tilde{U}(\cdot,\mu)$ are easily determined, and then we approximate $\tilde{L}(x)$ and $\tilde{U}(x)$ with $\min_{\mu \in M} \tilde{L}(x,\mu)$ and $\max_{\mu \in M} \tilde{U}(x,\mu)$, respectively.

For our numerical example with sample size $n = 500$, the bands $(\tilde{L},\tilde{U})$ differ only little from $(L,U)$, but for sample size $n = 2500$, the improvement is substantial; see Figures~\ref{fig:CB50_N2500_S}, \ref{fig:CB25_N2500_S} and \ref{fig:CB75_N2500_S}. The original band $(L,U)$ is depicted with thin black lines, and the refined band $(\tilde{L},\tilde{U})$ is depicted with blue lines.

\begin{figure}
\centering
\includegraphics[width=0.9\textwidth]{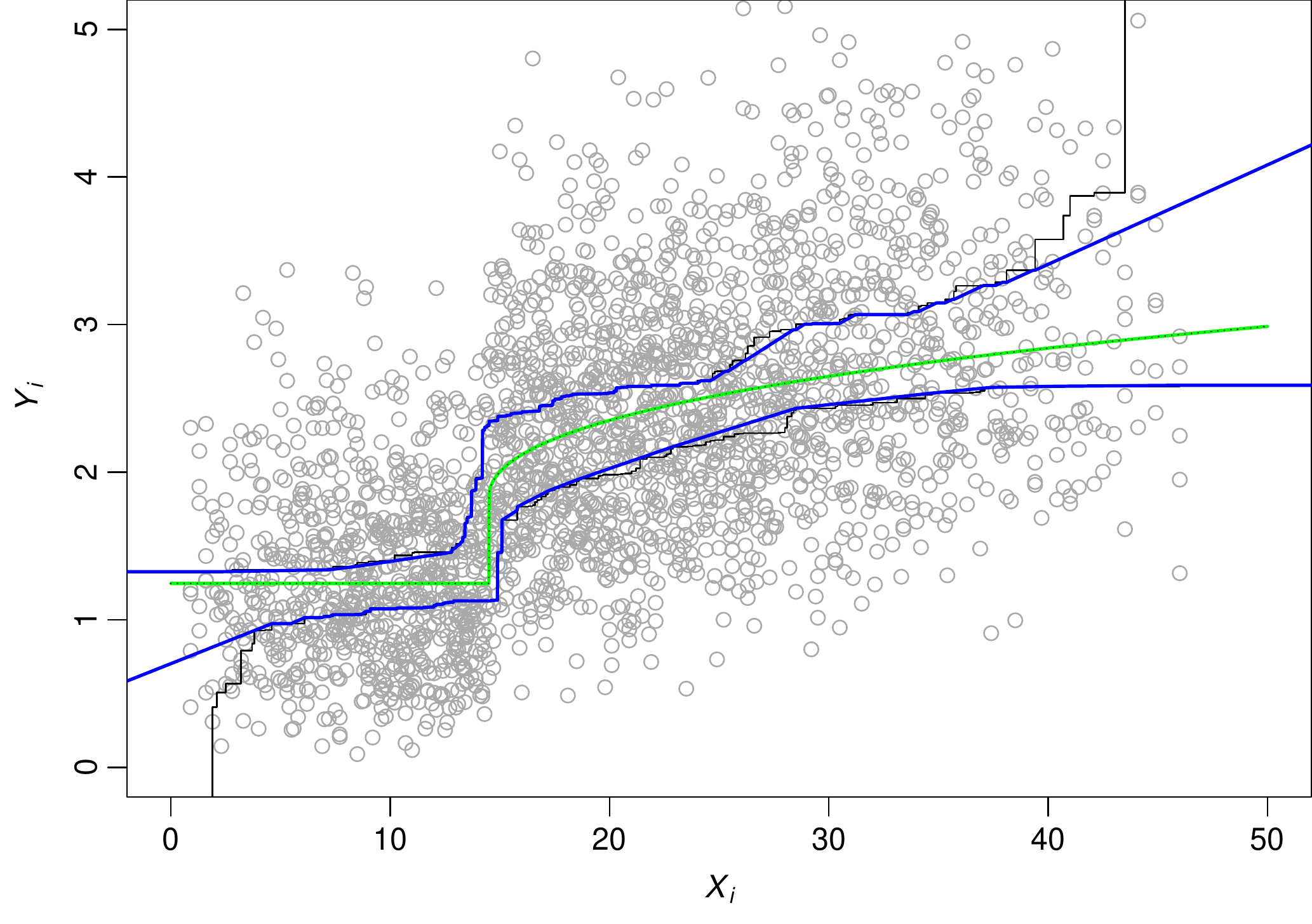}
\caption{$95\%$-Confidence band for S-shaped $Q_{0.5}$ for the data in Figure~\ref{fig:Data_N2500}, together with the true quantile function $Q_{0.5}$ (green, dotted).}
\label{fig:CB50_N2500_S}
\end{figure}

\begin{figure}
\centering
\includegraphics[width=0.9\textwidth]{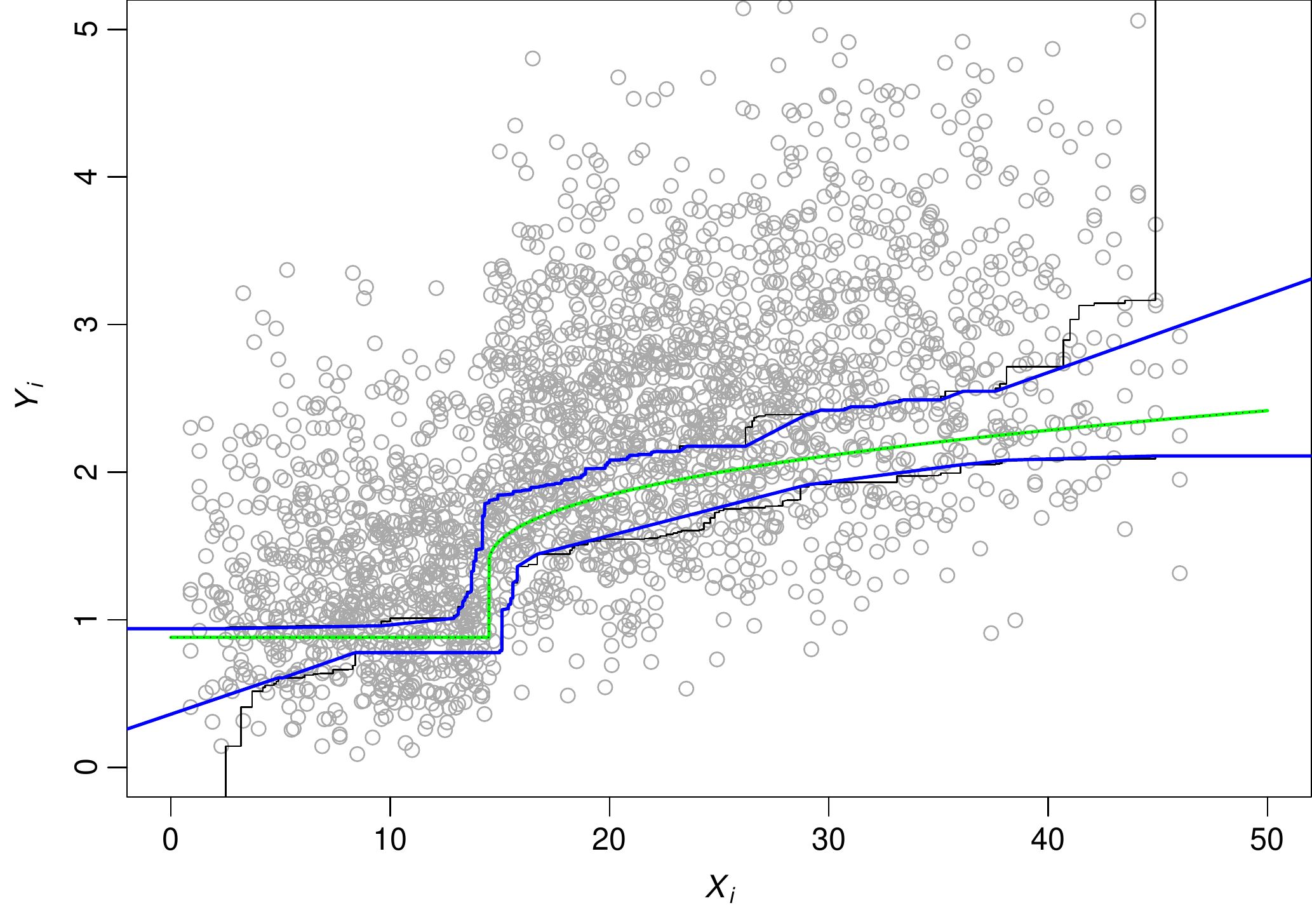}
\caption{$95\%$-Confidence band for S-schaped $Q_{0.25}$ for the data in Figure~\ref{fig:Data_N2500}, together with the true quantile function $Q_{0.25}$ (green, dotted).}
\label{fig:CB25_N2500_S}
\end{figure}

\begin{figure}
\centering
\includegraphics[width=0.9\textwidth]{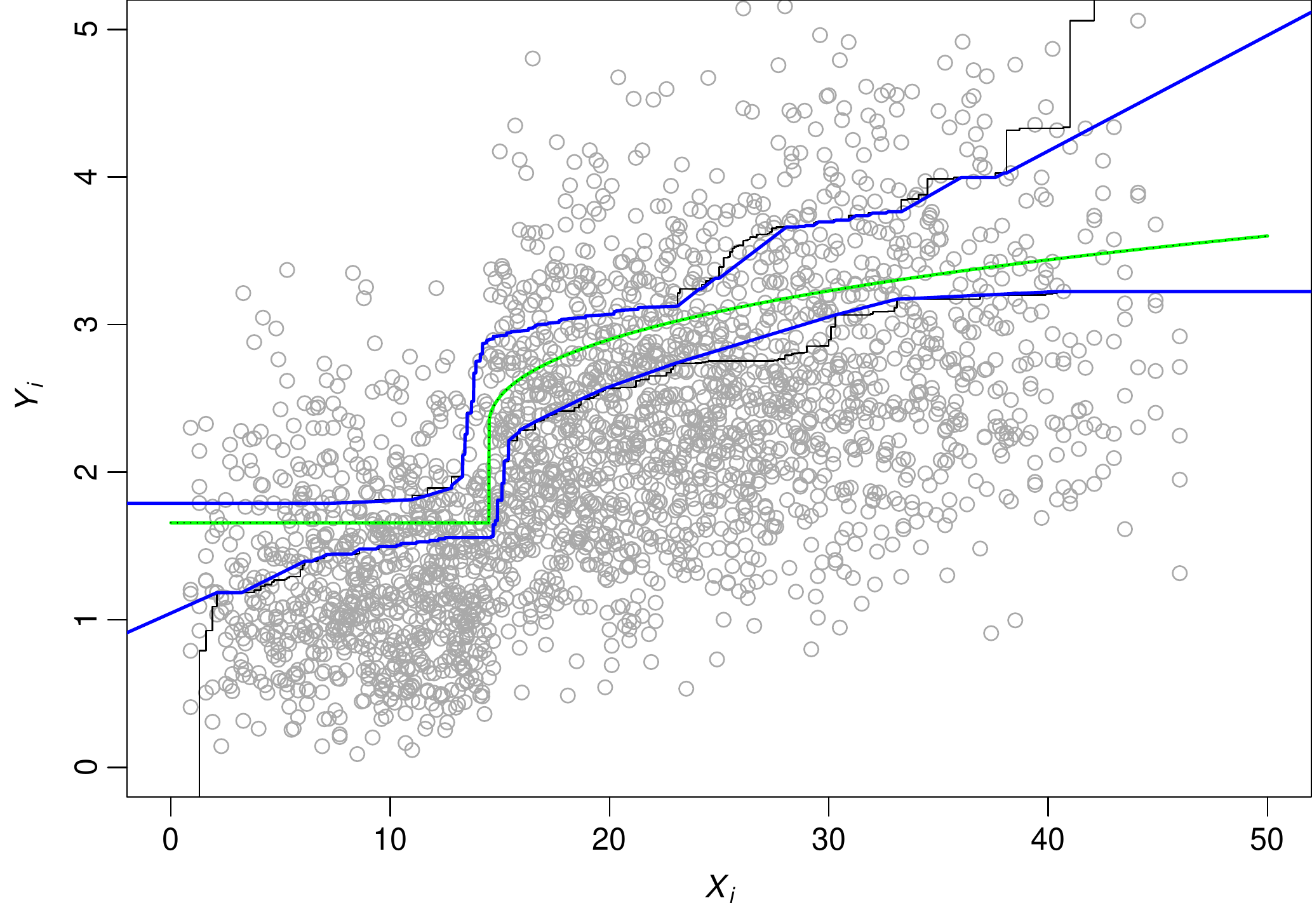}
\caption{$95\%$-Confidence band for S-shaped $Q_{0.75}$ for the data in Figure~\ref{fig:Data_N2500}, together with the true quantile function $Q_{0.75}$ (green, dotted).}
\label{fig:CB75_N2500_S}
\end{figure}


\end{document}